\documentclass[12pt]{amsart}
\usepackage{amsmath}
\usepackage{hyperref}
\usepackage{amssymb}
\usepackage{latexsym}
\usepackage{amscd}
\usepackage{graphicx} %We can use any other package if necessary
\usepackage{amsthm}
\usepackage{mathrsfs}
\usepackage{xypic}
\usepackage{bm}

\newdimen\AAdi%
\newbox\AAbo%
\def\AAk#1#2{\s_etbox\AAbo=\hbox{#2}\AAdi=\wd\AAbo\kern#1\AAdi{}}%
\def\AAr#1#2#3{\s_etbox\AAbo=\hbox{#2}\AAdi=\ht\AAbo\raise#1\AAdi\hbox{#3}}%
\font\tenmsb=msbm10 at 12pt \font\sevenmsb=msbm7 at 8pt
\font\fivemsb=msbm5 at 6pt
\newfam\msbfam
\textfont\msbfam=\tenmsb \scriptfont\msbfam=\sevenmsb
\scriptscriptfont\msbfam=\fivemsb
\def\Bbb#1{{\tenmsb\fam\msbfam#1}}
\newtheorem{thm}{Theorem}[section]
\newtheorem{lem}{Lemma}[section]

\newtheorem{pro}{Proposition}[section]

\newcommand{\ba}{\begin{array}}
\newcommand{\ea}{\end{array}}

\newcommand{\Section}[2]{\setcounter{equation}{0}
\allowdisplaybreaks
\section[#1]{#2}}

\def\n{\nabla}

\def\f#1#2{\frac{#1}{#2}}

\def\grs#1#2{\bold G_{#1,#2}}

\def\mc#1{\mathcal{#1}}

\def\a{\alpha}
\def\be{\beta}

\def\p#1{\partial #1}

\def\de{\delta}
\def\De{\Delta}

\def\ep{\varepsilon}

\def\G{\Gamma}
\def\g{\gamma}

\def\la{\lambda}
\def\La{\Lambda}
\def\om{\omega}
\def\Om{\Omega}
\def\th{\theta}

\def\si{\sigma}

\def\w{\wedge}

\def\R{\Bbb{R}}

\def\lan{\langle}
\def\ran{\rangle}
\def\ra{\rightarrow}

\def\ol{\overline}
\def\mb{\mathbf}

\def\Id{\mathbf{Id}}
\def\Arg{\text{Arg}}

\subjclass{58E20,53A10.}

\begin{document}
\title
[Spherical Bernstein theorem] {A spherical Bernstein theorem for
minimal submanifolds of higher codimension}

\author
[J. Jost, Y. L. Xin and Ling Yang]{J. Jost, Y. L. Xin and Ling Yang}
\address{Max Planck Institute for Mathematics in the
Sciences, Inselstr. 22, 04103 Leipzig, Germany.}
\email{jost@mis.mpg.de}
\address {Institute of Mathematics, Fudan University,
Shanghai 200433, China.} \email{ylxin@fudan.edu.cn}
\address{Institute of Mathematics, Fudan University,
Shanghai 200433, China.} \email{yanglingfd@fudan.edu.cn}
\thanks{The first author is supported by the ERC Advanced Grant
  FP7-267087. The second named author and the third named author are partially supported by
  NSFC. They are grateful to the Max Planck
Institute for Mathematics in the Sciences in Leipzig for its
hospitality and  continuous support. }

\begin{abstract}
Combining the tools of geometric analysis with  properties of Jordan angles and angle space distributions,
we derive a spherical  and a Euclidean Bernstein theorem for
minimal submanifolds of arbitrary dimension and codimension, under the condition that
the Gauss image is contained in some geometrically defined  closed region of a Grassmannian manifold. The proof depends
on the subharmoncity of an auxiliary function, the Codazzi equations
and  geometric measure theory.

\end{abstract}
\maketitle

\Section{Introduction}{Introduction}

This paper is a part of our systematic approach to the Bernstein problem in higher codimension. The Bernstein problem has a spherical and a Euclidean version, and the two are tightly related and essentially equivalent, as is well known and as we shall explain in a moment in more detail.

The Euclidean version says that a complete $n$-dimensional minimal submanifold $M$ of $\R^{n+m}$, that is, of codimension $m$, has to be an affine subspace if its Gauss image is contained in a sufficiently small subset of the Grassmann manifold $\grs{m}{n}$. Equivalently, it is affine when all of its normal spaces $N$ satisfies $\lan N,Q_0\ran >c_0$ 
 for a fixed reference space $Q_0$ and some positive constant
 $c_0$. In either formulation, we are assuming that the tangent, or
 equivalently, the normal spaces do not change their direction too
 much when we move across $M$. That some such condition is necessary
 follows from an example of Lawson-Osserman\cite{l-o} with $\lan
 N,Q_0\ran=1/9$. And that example tells us, more precisely, that the
 condition has to be stricter for $m>1$ than in the codimension 1 case, the setting of the classical Bernstein theorem \cite{be} and its extensions by Fleming \cite{f}, de Giorgi \cite{g}, Almgren
\cite{al}, Simons
\cite{Si}, Moser \cite{m}, and others.

The spherical Bernstein theorem concerns compact $(n-1)$-dimensional minimal submanifolds of the sphere $S^{n+m-1}$, and analogously, the aim is to prove that they are totally geodesic (i.e. equatorial) subspheres when their normal planes do not change their directions too much.

As indicated, we are interested here in the case $m>1$, and we ask what the optimal quantitative condition is. In previous work, we have shown
\begin{thm}\cite{j-x-y2}\label{t1}
Let $M$ be an $(n-1)$-dimensional compact minimal submanifold in $S^{n+m-1}$. Suppose that there is a fixed
oriented $m$-plane $Q_0$ and a number $c_0>1/3$, such that $\lan N,Q_0\ran\geq c_0$ holds for all
normal $m$-planes $N$ of $M$. Then $M$ is totally geodesic.
\end{thm}
Our question here is whether this is optimal, that is, whether there exists a counterexample for $c_0=1/3$, or whether one can move beyond. In this paper, we show that the result continues to hold for $c_0=1/3$, that is,
\begin{thm}\label{t2}
Let $M$ be an $(n-1)$-dimensional compact minimal submanifold in $S^{n+m-1}$. Suppose that there is a fixed
oriented $m$-plane $Q_0$, such that $\lan N,Q_0\ran\geq 1/3$ holds for all
normal $m$-planes $N$ of $M$. Then $M$ is totally geodesic.
\end{thm}
This might look like a small and insignificant step, but as in many examples of geometric analysis, limiting cases often are much harder than those involving strict inequalities. The reason is that one needs additional tools to analyze possible limit configurations before one can deduce that they can't exist after all. Typically, in such cases, the analytical estimates need to be supplemented by considerations of a more algebraic nature. That is also the case here. We shall carefully utilize the information contained in the Codazzi equations.
Still, there exists a quantitative gap between the above positive results and the counterexample of Lawson-Osserman, as the latter corresponds to the value $c_0=1/9$ instead of the $1/3$ that we can currently achieve. Nevertheless, we believe that since we can now rule out a counterexample for  $c_0=1/3$, we expect that one can go even beyond that value. Whether  $c_0=1/9$ is the largest value with a counterexample or whether a counterexample different from the Lawson-Osserman one exists for some value between $1/9$ and $1/3$, we currently do not know.

Let us now recall the relation between the spherical and the Euclidean Bernstein problem and then state our results for the latter. Fleming's idea \cite{f} was that by rescaling a  nontrivial minimal graph in Euclidean space, one obtains a nonflat minimal cone, and the intersection of that cone with the unit sphere then is a compact minimal submanifold of the latter. Therefore, conditions ruling out the latter can be translated into conditions ruling out the former. Because of the noncompact nature of minimal graphs in Euclidean space, as an important technical ingredient, one needs to invoke Allard's regularity theory. 

We then obtain the following result concerning the Euclidean version
of the higher codimensional Bernstein problem (see Theorem \ref{t3}), which is an improvement of Theorem 1.1 in \cite{j-x-y2}.

\begin{thm}\label{t3a}
Let $f:=(f^1,\cdots,f^m)$ be a smooth $\R^m$-valued function defined everywhere on $\R^n$. If its graph
$M:=\text{graph }f=\{(x,f(x)):x\in \R^m\}$ is a minimal submanifold in $\R^{n+m}$, and
\begin{equation}\label{con2}
\De_f:=\Big[\det\Big(\de_{ij}+\sum_\a \f{\p f^\a}{\p x^i}\f{\p
f^\a}{\p x^j}\Big)\Big]^{\f{1}{2}}\leq 3,
\end{equation}
then $f^1,\cdots,f^m$ have to be affine linear, that is, represent an affine $n$-plane in $\R^{n+m}$.

\end{thm}
The 3 here corresponds to the $1/3$ in the preceding results as will become clear below when we describe the geometry.

In geometric terms, we approach the Bernstein problem via the normal
Gauss map $\g$ of a minimal submanifold $M$ which takes values in the
Grassmannian manifold $\grs{m}{n}$. The strategy which was first
applied in \cite{h-j-w} then is to show that $\g$ is constant. This,
of course, is equivalent to $M$ being affine. The main point here is
that  this Gauss map is harmonic. Therefore, the strategy then leads
to showing that under appropriate geometric conditions, a harmonic map
into a Grassmannian $\grs{m}{n}$ has to be constant. This strategy
could be successfully applied in increasing generality, in
\cite{h-j-w,j-x,j-x-y2,j-x-y3}. Essentially, one tries to translate a
geometric restriction on the size of the Gauss image $\g(M)$ into the
subharmonicity of a composition function $\phi\circ \g$ where
$\phi:\g(M)\to \R$ is a suitable scalar function, for instance a
strictly convex function. When we are in the spherical setting, the
domain $M$ is compact, and so, such a  subharmonic function $\phi
\circ \g$ then has to be constant by the maximum principle. One then
utilizes this to deduce that $\g$ is constant itself. In the Euclidean
setting, $\phi \circ \g$ has to be subjected to clever and subtle
estimates to eventually reach the same conclusion. As already
mentioned, for the case $c_0=1/3$ which we are treating in the present
paper, this strategy by itself is not yet powerful enough and needs to
be supplemented by detailed algebraic considerations. In more precise
terms, we consider eigenspaces for Jordan angles between the tangent
or normal space of our minimal submanifold $M$ and some fixed reference
space. Invoking a theorem of Nomizu \cite{no}, we find that,
generically, they yield smooth subbundles of the tangent or normal
bundle of $M$. This will provide us with decompositions that can be
algebraically exploited.

The paper is organized as follows. In Section \ref{cja}, the basic geometric concepts on Grassmannian manifolds, those  of Jordan angles, angle spaces, multiplicities, anti-involutive automorphisms and
angle space distributions,  are introduced, which will play an important role in our statement, and the connection between the
$w$-function on Grassmannian manifolds and the Jordan angles is revealed. Let $Q_0$ be a fixed point on $\grs{m}{n}$,
then the composition of $w(\cdot,Q_0)$ and the normal Gauss map yields a smooth function on $M$, an arbitrary submanifold
in $\R^{n+m}$, which is called the $w$-function. Based on \cite{fc}, \cite{j-x-y2} and the properties of Jordan angles, one can prove $v:=w^{-1}$ is
a subharmonic function whenever $M$ has parallel mean curvature and $v\leq 3$. We also explore the second fundamental form of $M$ provided that $v\leq 3$, $\De v\equiv 0$ and $|B|^2\neq 0$
everywhere, and discover that $M$ is a simple austere submanfold,  a notion  introduced by Harvey-Lawson \cite{h-l} and further studied by Bryant \cite{br}. These are the main points
of Section \ref{pre}. Finally in Section \ref{be}, combining the Codazzi equations and basic properties of angle spaces, we derive
a spherical Bernstein theorem, and also the corresponding Euclidean Bernstein theorem.

\bigskip\bigskip
\Section{Jordan angles and angle space distributions}{Jordan angles and angle space distributions}
\label{cja}

Let $\R^{n+m}$ be an $(n+m)$-dimensional Euclidean space.
The oriented $m$-spaces in $\R^{n+m}$ constitute the Grassmann manifold $\grs{m}{n}$, which is the Riemannian symmetric space
of compact type $SO(n+m)/ SO(n)\times SO(m)$.

Let $P, Q_0$ be 2 points in $\grs{m}{n}$. The \textit{Jordan angles} between $P$ and
$Q_0$ are the critical values of the angle $\th$
between a nonzero vector $u$ in $P$ and its orthogonal projection
$u^*$ in $Q_0$ as $u$ runs through $P$. If $\th$ is a nonzero
Jordan angle between $P$ and $Q_0$ determined by a unit vector $u$
in $P$ and its projection $u^*$ in $Q_0$, then $u$ is called an
\textit{angle direction} of $P$ relative to $Q_0$, and the $2$-plane
spanned by $u$ and $u^*$ is called an \textit{angle 2-plane} between
$P$ and $Q_0$ (see \cite{w}).

Now, we give a slightly different description for the Jordan angles, which will be useful later.

Denote by $\mc{P}_0$ the orthogonal projection of $\R^{n+m}$ onto
$Q_0$ and by $\mc{P}$  the orthogonal projection of $\R^{n+m}$
onto $P$. Then for an arbitrary vector $u\in P$ and $\ep\in Q_0$,
\begin{equation}
\aligned
\lan \mc{P}_0 u,\ep\ran&=\lan \mc{P}_0 u+(u-\mc{P}_0 u),\ep\ran=\lan u,\ep\ran\\
&=\lan u,\mc{P}\ep+(\ep-\mc{P}\ep)\ran=\lan u,\mc{P}\ep\ran.
\endaligned
\end{equation}
This means that $\mc{P}$ is adjoint to $\mc{P}_0$ with respect to the
canonical Euclidean inner product. Moreover,
\begin{equation}
\lan (\mc{P}\circ\mc{P}_0)u,v\ran=\lan \mc{P}_0u,\mc{P}_0v\ran=\lan
u,(\mc{P}\circ \mc{P}_0)v\ran
\end{equation}
holds for any $u,v\in P$, which implies that $\mc{P}\circ \mc{P}_0$ is a
nonnegative definite self-adjoint transformation on $P$.

For every nonzero $u\in P$,
\begin{equation}
\cos^2 \angle(u,u^*)=\f{\lan u^*,u^*\ran}{\lan u,u\ran}=\f{\lan
\mc{P}_0 u,\mc{P}_0 u\ran}{\lan u,u\ran}=\f{\lan (\mc{P}\circ
\mc{P}_0)u,u\ran}{\lan u,u\ran}.
\end{equation}
Hence $\th$ is a Jordan angle between $P$ and $Q_0$ if and only if
$\mu:=\cos^2\th$ is an eigenvalue of $\mc{P}\circ \mc{P}_0$, and $u$
is an angle direction associated to $\th$ if and only if it is
an eigenvector associated to the eigenvalue $\mu$, i.e.
\begin{equation}
(\mc{P}\circ \mc{P}_0)u=\mu u=\cos^2\th\ u.
\end{equation}
 Therefore, all the angle directions associated to $\th$ constitute a linear subspace of $P$, which is called
an \textit{angle space} of $P$ relative to $Q_0$ and we denote it by
$P_\th$. The dimension of $P_\th$ is called the
\textit{multiplicity} of $\th$, which is denoted by $m_\th$. If we
denote by $\text{Arg}(P,Q_0)$ the set consisting of all the Jordan
angles between $P$ and $Q_0$, then
\begin{equation}\label{oplus1}
P=\bigoplus_{\th\in \text{Arg}(P,Q_0)}P_\th.
\end{equation}
and hence
\begin{equation}
m=\sum_{\th\in \text{Arg}(P,Q_0)}m_\th.
\end{equation}
The angle spaces are mutually orthogonal to each other, and in
particular
\begin{equation}
P_0=P\cap Q_0,\qquad P_{\pi/2}=P\cap Q_0^\bot.
\end{equation}

Similarly, $\th$ is a Jordan angle between $Q_0$ and $P$ if and only
if $\mu:=\cos^2\th$ is an eigenvalue of $\mc{P}_0\circ \mc{P}$.
Denote by $(Q_0)_\th$ the angle space of $Q_0$ relative to $P$ associated to $\th$, then
$\ep\in (Q_0)_\th$ if and only if $(\mc{P}_0\circ
\mc{P})\ep=\cos^2\th\ \ep$, and
\begin{equation}\label{oplus2}
Q_0=\bigoplus_{\th\in \text{Arg}(Q_0,P)}(Q_0)_\th
\end{equation}
with $\text{Arg}(Q_0,P)$ denoting the set consisting of all the
Jordan angles between $Q_0$ and $P$.

Let $P^\bot$ and $Q_0^\bot$ be the orthogonal complements of $P$ and
$Q_0$, and denote by $\mc{P}_0^\bot$ and $\mc{P}^\bot$ the
orthogonal projections of $\R^{n+m}$ onto $P^\bot$ and $Q_0^\bot$,
respectively. As above, the set consisting of all the
Jordan angles between $P^\bot$ and $Q_0^\bot$ is denoted by
$\text{Arg}(P^\bot,Q_0^\bot)$, $P_\th^\bot$ denotes the angle space
associated to $\th\in \text{Arg}(P^\bot,Q_0^\bot)$, and
$m_\th^\bot:=\dim P_\th^\bot$ denotes the multiplicity of $\th$.
 The following lemma
reveals the close relationship between $\text{Arg}(P,Q_0)$, $\Arg(Q_0,P)$ and
$\text{Arg}(P^\bot,Q_0^\bot)$.

\bigskip

\begin{lem}\label{Jordan}

Let $P,Q_0\in \grs{m}{n}$, then $\Arg(P,Q_0)=\Arg(Q_0,P)$ and the multiplicities of each corresponding Jordan angles are equivalent.
If we denote
\begin{equation}
R_\th:=P_\th+(Q_0)_\th
\end{equation}
for each $\th\in \Arg(P,Q_0)$,
then $R_\th\bot R_\si$ whenever $\th\neq \si$, and
\begin{equation}\label{Rth}
P+Q_0=\bigoplus_{\th\in \Arg(P,Q_0)}R_\th.
\end{equation}
For any $\th\in (0,\pi/2]$, $\th\in \Arg(P^\bot,Q_0^\bot)$ if and only if
 $\th\in \Arg(P,Q)$, and $m_\th^\bot=m_\th$, $R_\th=P_\th\oplus P_\th^\bot$. Moreover, for every $\th\in
\text{Arg}(P,Q_0)\cap (0,\pi/2)$, there exists $\Phi_\th: R_\th\ra R_\th$, satisfying

(i) $|\Phi_\th(\xi)|=|\xi|$ for every $\xi\in R_\th$;

(ii) $\Phi_\th^2=-\Id$;

(iii) $\Phi_\th(P_\th)=P_\th^\bot$, $\Phi_\th(P_\th^\bot)=P_\th$;

(iv) For any nonzero vector $u\in P_\th$ ($v\in P_\th^\bot$),
$\Phi_\th(u)$ ($\Phi_\th(v)$) lies in the angle 2-plane generated by
$u$ ($v$); more precisely,
\begin{equation}\label{Phi}\aligned
\sec\th\ \mc{P}_0 u&=\cos\th\ u-\sin\th\ \Phi_\th(u),\\
\sec\th\ \mc{P}_0^\bot v&=\cos\th\ v-\sin\th\ \Phi_\th(v).
\endaligned
\end{equation}
$\Phi_\th$ is called the \textbf{anti-involutive automorphism} associated to $\th$.
\end{lem}

\begin{proof}

Given any $\th\in
\text{Arg}(P,Q_0)\cap [0,\pi/2)$ and $u\in P_\th$, letting
$\ep:=\mc{P}_0 u$ gives
\begin{equation}
(\mc{P}_0\circ \mc{P})\ep=(\mc{P}_0\circ \mc{P}\circ
\mc{P}_0)u=\cos^2 \th\ \mc{P}_0 u=\cos^2\th\ \ep.
\end{equation}
Hence $\th\in \text{Arg}(Q_0,P)$ and $\mc{P}_0$ is an injective
linear mapping from $P_\th$ into $(Q_0)_\th$. On the other hand,
assuming $\th\in \text{Arg}(Q_0,P)\cap [0,\pi/2)$, one can deduce
$\th\in \text{Arg}(P,Q_0)$ and $\mc{P}$ is an injective linear
mapping from $(Q_0)_\th$ into $P_\th$. Thus both $\mc{P}_0$ and
$\mc{P}$ are linear isomorphisms whenever $\th\in [0,\pi/2)$, and then
$\dim (Q_0)_\th=\dim P_\th$. In
conjunction with (\ref{oplus1}) and (\ref{oplus2}),
\begin{equation}\label{orth}\aligned
\dim(Q_0)_{\pi/2}&=m-\sum_{\th\in \Arg(Q_0,P)\cap [0,\pi/2)}\dim (Q_0)_\th\\
&=m-\sum_{\th\in \Arg(P,Q_0)\cap [0,\pi/2)}\dim P_\th\\
&=\dim P_{\pi/2}.
\endaligned
\end{equation}
This mean $\pi/2\in \Arg(Q_0,P)$ if and only if $\pi/2\in \Arg(P,Q_0)$ and the multiplicities are equivalent. Therefore,
$\text{Arg}(Q_0,P)=\text{Arg}(P,Q_0)$ and $\dim (Q_0)_\th=\dim
P_\th=m_\th$ for every $\th\in \text{Arg}(P,Q_0)$.

Let $\th,\si$ be distinct Jordan angles between $P$ and $Q_0$, it has been shown above that
$P_\th\bot P_\si$ and $(Q_0)_\th\bot (Q_0)_\si$. For any $u\in P_\th$ and $\ep\in (Q_0)_\si$,
$$\lan u,\ep\ran=\lan u,\mc{P}(\ep)\ran \in \lan P_\th,P_\si\ran=(0),$$
which implies $P_\th\bot (Q_0)_\si$ and similarly $(Q_0)_\th\bot P_\si$. Hence $R_\th\bot R_\si$ and then
(\ref{Rth}) immediately follows from (\ref{oplus1}) and (\ref{oplus2}).

If $\th\in (0,\pi/2)$ is a Jordan angle between $P$ and $Q_0$, then
for any nonzero vector $u\in P_\th$, we have $|\mc{P}_0 u|=\cos\th
|u|$ and $(\mc{P}\circ \mc{P}_0)u=\cos^2 \th\ u$. Denote
\begin{equation}\label{ep1}
\ep:=\sec\th\ \mc{P}_0 u,
\end{equation}
then $|\ep|=|u|$. Now we put
\begin{equation}
v:=-\sec\th\csc\th(\mc{P}^\bot\circ \mc{P}_0)u\in P^\bot,
\end{equation}
then
\begin{equation}\label{Phi1}\aligned
\ep&=\sec\th\ \mc{P}_0 u=\sec\th\big[(\mc{P}\circ \mc{P}_0)u+(\mc{P}^\bot\circ \mc{P}_0)u\big]\\
&=\cos\th\ u-\sin\th\ v
\endaligned
\end{equation}
which implies $v\in \text{span}\{u,\mc{P}_0u\}\subset R_\th$ and $|v|=|u|$. Since
\begin{equation*}
\aligned
(\mc{P}^\bot\circ \mc{P}_0^\bot)v&=-\sec\th\csc\th (\mc{P}^\bot\circ \mc{P}_0^\bot \circ (\Id-\mc{P})\circ \mc{P}_0) u\\
&=\sec\th\csc\th(\mc{P}^\bot\circ (\Id-\mc{P}_0)\circ \mc{P}\circ\mc{P}_0) u\\
&=\cot\th(\mc{P}^\bot\circ (\Id-\mc{P}_0))u\\
&=-\cot\th(\mc{P}^\bot\circ \mc{P}_0)u=\cos^2\th\ v,
\endaligned
\end{equation*}
$\th\in \text{Arg}(P^\bot,Q_0^\bot)$, $v\in P_\th^\bot$ and
\begin{equation}\label{Phi5}
\Phi_{1,\th}: u\mapsto -\sec\th\csc\th(\mc{P}^\bot\circ \mc{P}_0)u
\end{equation}
is injective from $P_\th$ into $P_\th^\bot$, which
implies $m_\th^\bot\geq m_\th$.
(\ref{Phi1}) then becomes 
\begin{equation}\label{Phi3}
\sec\th\ \mc{P}_0 u=\cos\th\ u-\sin\th\ \Phi_{1,\th}(u).
\end{equation}

Similarly, if $\th\in (0,\pi/2)$ is a Jordan angle between $P^\bot$
and $Q_0^\bot$, then $\th\in \text{Arg}(P,Q_0)$ and
\begin{equation}\label{Phi2}
\Phi_{2,\th}:v\mapsto -\sec\th\csc\th(\mc{P}\circ \mc{P}_0^\bot)v
\end{equation}
is also injective from $P_\th^\bot$ into $P_\th$,
which implies
\begin{equation}\label{Phi4}
\sec\th\ \mc{P}_0^\bot v=\cos\th\ v-\sin\th\ \Phi_{2,\th}(v)
\end{equation}
and $m_\th\geq m_\th^\bot$.
Therefore $m_\th=m_\th^\bot$, and both $\Phi_{1,\th}$ and $\Phi_{2,\th}$
are isomorphisms.

For any $\th\in \Arg(P,Q_0)\cap (0,\pi/2)$, $P_\th^\bot=\Phi_{1,\th}(P_\th)\subset R_\th$
gives $P_\th\oplus P_\th^\bot\subset R_\th$. On the other hand,
$P_\th\cap (Q_0)_\th\subset (P\cap Q_0)\cap P_\th=P_0\cap P_\th=(0)$
implies $\dim R_\th=\dim P_\th+\dim (Q_0)_\th=2m_\th=\dim P_\th+\dim P_\th^\bot$, then $R_\th=P_\th\oplus P_\th^\bot$. Now we define $\Phi_\th:
R_\th\ra R_\th$
\begin{equation}
u+v\mapsto \Phi_{1,\th}(u)+\Phi_{2,\th}(v)\qquad \forall u\in P_\th, v\in P_\th^\bot.
\end{equation}
Then $\Phi_\th$ is an isometric automorphism,
$\Phi_\th(P_\th)=P_\th^\bot$, $\Phi_\th(P_\th^\bot)=P_\th$ and (\ref{Phi}) immediately follows from
(\ref{Phi3}) and (\ref{Phi4}). For any
$u\in P_\th$, (\ref{Phi5}) and (\ref{Phi2}) implies 
$$\aligned
\Phi_\th^2(u)&=\sec^2\th \csc^2\th (\mc{P}\circ \mc{P}_0^\bot\circ \mc{P}^\bot\circ \mc{P}_0)u\\
&=\sec^2\th \csc^2\th (\mc{P}\circ \mc{P}_0^\bot\circ (\Id-\mc{P})\circ\mc{P}_0)u\\
&=-\sec^2\th \csc^2\th(\mc{P}\circ \mc{P}_0^\bot \circ \mc{P}\circ \mc{P}_0)u\\
&=-\csc^2\th (\mc{P}\circ \mc{P}_0^\bot)u=\csc^2\th(\mc{P}\circ (\Id-\mc{P}_0)u)\\
&=-u
\endaligned$$
and similarly $\Phi_\th^2(v)=-v$ for each $v\in P_\th^\bot$. Hence
$\Phi_\th^2=-\Id$.

It remains to prove $m_{\pi/2}=m_{\pi/2}^\bot$. By (\ref{orth}),
\begin{equation}
\aligned
m_{\pi/2}^\bot&=\dim P_{\pi/2}^\bot=\dim (P^\bot\cap Q_0)\\
&=\dim (Q_0)_{\pi/2}=\dim P_{\pi/2}\\
&=m_{\pi/2}.
\endaligned
\end{equation}

\end{proof}

Denote
\begin{equation}
r:=\sum_{\th\in \text{Arg}(P,Q_0)\cap
(0,\pi/2]}m_\th=\sum_{\th\in \text{Arg}(P^\bot,Q_0^\bot)\cap
(0,\pi/2]}m_\th^\bot
\end{equation}
then $0\in \text{Arg}(P,Q_0)$ if and only if $r<m$, and
$m_0=m-r$. Similarly $0\in \text{Arg}(P^\bot,Q_0^\bot)$ if and
only if $r<n$, and $m_0^\bot=n-r$.

It is well-known that $\grs{m}{n}$ can be viewed as a
submanifold of some Euclidean space via the Pl\"ucker embedding. The
restriction of the Euclidean inner product on $\grs{m}{n}$ is
denoted by $w:\grs{m}{n}\times \grs{m}{n}\ra \R$. Let
$\{\ep_1,\cdots,\ep_m\}$ and $\{u_1,\cdots,u_m\}$ denote oriented orthonormal bases of $Q_0$ and $P$,
respectively, then
\begin{equation}
\aligned
w(P,Q_0)&=\lan u_1\w\cdots\w u_m,\ep_1\w \cdots\w \ep_m\ran\\
&=\left|\begin{array}{ccc}
\lan u_1,\ep_1\ran & \cdots & \lan u_1,\ep_m\ran\\
                   & \cdots &                   \\
\lan u_m,\ep_1\ran & \cdots & \lan u_m,\ep_m\ran
\end{array}\right|\\
&=\left|\begin{array}{ccc}
\lan \mc{P}_0u_1,\ep_1\ran & \cdots & \lan \mc{P}_0u_1,\ep_m\ran\\
                   & \cdots &                   \\
\lan \mc{P}_0u_m,\ep_1\ran & \cdots & \lan \mc{P}_0u_m,\ep_m\ran
\end{array}\right|\\
&=\lan \mc{P}_0 u_1\w\cdots\w \mc{P}_0 u_m,\ep_1\w \cdots\w
\ep_m\ran.
\endaligned
\end{equation}
Note that the definition of $w(P,Q_0)$ does not depend on the
choices of oriented orthonormal bases of $P$ and $Q_0$.

As shown above, one can choose an orthonormal basis
$\{u_1,\cdots,u_m\}$ of $P$, such that $P=u_1\w\cdots\w u_m$,
and each $u_\a$ is an angle direction of $P$ relative to $Q_0$. More
precisely,
\begin{equation}
(\mc{P}\circ\mc{P}_0)u_\a=\cos^2\th_\a u_\a,
\end{equation}
where $\th_\a$ is a Jordan angle between $P$ and $Q_0$. Hence
\begin{equation}\label{basis}\aligned
\lan \mc{P}_0 u_\a,\mc{P}_0 u_\be\ran&=\lan \mc{P}_0 u_\a,u_\be\ran=\lan (\mc{P}\circ\mc{P}_0)u_\a,u_\be\ran\\
&=\cos^2\th_\a\lan u_\a,u_\be\ran=\cos^2\th_\a \de_{\a\be}.
\endaligned
\end{equation}

If $\pi/2\in \text{Arg}(P,Q_0)$, then there exists $\a$, such that
$\th_\a=\pi/2$, which implies $\mc{P}_0 u_\a=0$, and moreover
\begin{equation}
w(P,Q_0)=\lan \mc{P}_0 u_1\w \cdots\w \mc{P}_0 u_m,\ep_1\w\cdots\w
\ep_m\ran=0.
\end{equation}
Otherwise, by (\ref{basis}), $\{\ep'_\a=\f{\mc{P}_0 u_\a}{|\mc{P}_0
u_\a|}=\sec\th_\a (\mc{P}_0 u_\a):1\leq \a\leq m\}$ is an
orthonormal basis of $Q_0$, thus $Q_0=\ep'_1\w\cdots\w \ep'_m$ or
$-\ep'_1\w\cdots\w \ep'_m$. If $Q_0=\ep'_1\w\cdots\w \ep'_m$, then
\begin{equation}\aligned
w(P,Q_0)&=\lan \mc{P}_0u_1\w\cdots\w\mc{P}_0 u_m,\ep'_1\w\cdots\w\ep'_m\ran\\
&=\lan \cos\th_1\ep'_1\w\cdots\w \cos\th_m\ep'_m,\ep'_1\w\cdots\w \ep'_m\ran\\
&=\prod_{\a} \cos\th_\a \lan
\ep'_1\w\cdots\w\ep'_m,\ep'_1\w\cdots\w\ep'_m\ran=\prod_{\a}
\cos\th_\a;
\endaligned
\end{equation}
otherwise $Q_0=-\ep'_1\w\cdots\w \ep'_m$ and a similar calculation
shows
\begin{equation}
w(P,Q_0)=-\prod_{\a} \cos\th_\a.
\end{equation}
In summary,
\begin{equation}\label{w}
|w(P,Q_0)|=\prod_{\a} \cos\th_\a
\end{equation}
and $w(P,Q_0)>0$ if and only if $\mc{P}_0|_P$ is an orientation
preserving map from $P$ onto $Q_0$.
\bigskip

\begin{lem}\label{basis3}
If $w(P,Q_0)>0$, let
$$\pi/2>\th_1\geq \cdots\geq \th_r>\th_{r+1}=\cdots =\th_m=0$$
be the Jordan angles between $P$ and $Q_0$, then there exist an orthonormal basis
$\{\ep_1,\cdots,\ep_m\}$ of $Q_0$, an orthonormal basis $\{u_1,\cdots,u_m\}$ of
$P$ and an orthonormal basis $\{v_1,\cdots,v_n\}$ of $P^\bot$, such that
$Q_0=\ep_1\w\cdots\w \ep_m$, $P=u_1\w\cdots\w u_m$, $\mc{P}_0 u_\a=\cos\th_\a \ep_\a$
for each $1\leq \a\leq m$ and $\mc{P}_0 v_i=-\sin \th_i \ep_i$ for each $1\leq i\leq n$. Here we additionally assume
$\th_i:=0$ and $\ep_i:=0$ for any $m+1\leq i\leq n$. In particular, for each $1\leq \a\leq r$, $v_\a=\Phi_{\th_\a}(u_\a)$,
with $\Phi_{\th_\a}$ denoting the anti-involutive automorphism associated to $\th_\a$.
\end{lem}

\begin{proof}
Let $\{u_1,\cdots,u_m\}$ be an orthonormal basis of $P$, such that $P=u_1\w\cdots\w u_m$ and
$(\mc{P}\circ \mc{P}_0)u_\a=\cos^2 \th_\a u_\a$. Putting
\begin{equation}
\ep_\a:=\f{\mc{P}_0 u_\a}{|\mc{P}_0 u_\a|}=\sec \th_\a (\mc{P}_0 u_\a)\in Q_0\qquad \forall 1\leq \a\leq m,
\end{equation}
then (\ref{basis}) gives $\lan \ep_\a,\ep_\be\ran=\de_{\a\be}$ and hence $\{\ep_1,\cdots,\ep_m\}$ is an orthonormal
basis of $Q_0$. $w(P,Q_0)>0$ implies that $\mc{P}_0|_P$ is an orientation preserving map, thus $Q_0=\ep_1\w\cdots\w\ep_m$.

For each $1\leq \a\leq r$, let $v_\a:=\Phi_{\th_\a}(u_\a)$.
By Lemma \ref{Jordan},
\begin{equation}\label{basis2}
\ep_\a=\sec\th_\a(\mc{P}_0 u_\a)=\cos\th_\a u_\a-\sin\th_\a v_\a
\end{equation}
and
\begin{equation}
\lan v_\a,v_\be\ran=\de_{\a\be}\qquad \forall 1\leq
\a,\be\leq r.
\end{equation}
Let $\{v_{r+1},\cdots,v_n\}$ be an orthonormal basis of $P_0^\bot=P^\bot\cap Q_0^\bot$, then $\{v_1,\cdots,v_n\}$ is an orthonormal basis of
$P^\bot$. Applying (\ref{basis2}) gives
$$\aligned
\ep_\a&=\mc{P}_0 \ep_\a=\cos\th_\a(\mc{P}_0 u_\a)-\sin\th_\a(\mc{P}_0 v_\a)\\
&=\cos^2\th_\a\ep_\a-\sin\th_\a(\mc{P}_0 v_\a),
\endaligned
$$
i.e.
\begin{equation}
\mc{P}_0 v_\a=-\sin\th_\a\ep_\a\qquad \forall 1\leq \a\leq r.
\end{equation}
On the other hand, $v_i\in Q_0^\bot$ implies $\mc{P}_0 v_i=0$ for
each $r+1\leq i\leq n$. Hence
\begin{equation}
\mc{P}_0 v_i=-\sin\th_i \ep_i \qquad \forall 1\leq i\leq n.
\end{equation}

\end{proof}

Let $M$ be an $n$-dimensional submanifold in $\R^{n+m}$ and $Q_0$ be an fixed
$m$-dimensional subspace in $\R^{n+m}$. Denote by $TM$ and $NM$
the tangent bundle and the normal bundle along $M$, respectively.
Let $\th$ be a $[0,\pi/2]$-valued smooth function on $M$, if
$\th(p)\in \Arg(N_p M, Q_0)$ ($\th(p)\in \Arg(T_p M,Q_0^\bot)$),
we say $\th$ is a \textit{normal (tangent) Jordan angle function}  of $M$ relative to $Q_0$. Denote by $\Arg^N$ ($\Arg^T$)
the set consisting of all the normal (tangent) Jordan angle functions of
$M$ relative to $Q_0$. If $\th$ is a smooth function on $M$ that is nonzero
everywhere, then Lemma \ref{Jordan} implies $\th\in \Arg^N$ if and only if $\th\in \Arg^T$.

Denote
\begin{equation}\aligned
N_\th M&:=\{\nu\in N_p M: p\in M, \nu\text{ is an angle direction associated to }\th(p)\},\\
T_\th M&:=\{v\in T_p M: p\in M, v\text{ is an angle direction associated to }\th(p)\}.
\endaligned
\end{equation}
Let $\mc{P}_0$ and $\mc{P}_0^\bot$ be orthogonal projections onto $Q_0$ and $Q_0^\bot$,
$(\cdot)^T$ and $(\cdot)^N$ denote orthogonal projections onto $T_p M$ and $N_p M$, respectively.
Then $\nu\in N_\th M\cap N_p M$ if and only if
\begin{equation}
(\mc{P}_0 \nu)^N=\cos^2\th(p) \nu
\end{equation}
and similarly $u\in T_\th M\cap T_p M$ if and only if
\begin{equation}
(\mc{P}_0^\bot u)^T=\cos^2\th(p) u.
\end{equation}
Let $m_\th^N(p):=\dim (N_\th M\cap N_p M)$, $m_\th^T(p):=\dim (T_\th M\cap T_p M)$ for every $p\in M$, then
$m_\th^N$ and $m_\th^T$ are both $\Bbb{Z}^+$-valued functions on $M$.

\begin{lem}\label{dis}
Let $\th$ be a normal (tangent) Jordan angle function of $M$ relative to $Q_0$.
If $m_\th^N$ ($m_\th^T$) is a constant function on $M$, then $N_\th M$ ($T_\th M$) is a smooth subbundle of $NM$ ($TM$).
In this case, $N_\th M$ ($T_\th M$) is said to be a \textbf{normal (tangent) angle space distribution} associated to $\th$.
\end{lem}

\begin{proof}
We only give the proof for the tangential case, because it is similar for the normal case.

For any $p_0\in M$, let $U$ be a coordinate chart around $p_0$ and $\{e_1,\cdots,e_n\}$ be a orthonormal tangent frame field on $U$.
Denote
\begin{equation}
A_{ij}(p)=\lan \mc{P}_0^\bot e_i,e_j\ran\qquad \forall 1\leq i,j\leq n,
\end{equation}
then
$$\aligned
A_{ij}(p)&=\lan \mc{P}_0^\bot e_i,e_j\ran=\lan \mc{P}_0^\bot e_i,\mc{P}_0^\bot e_j\ran\\
&=\lan \mc{P}_0^\bot e_j,\mc{P}_0^\bot e_i\ran=\lan \mc{P}_0^\bot e_j,e_i\ran=A_{ji}(p)
\endaligned$$
and hence $p\mapsto A(p)$ is a smooth mapping from $U$ into $S_n$, the set of all $n\times n$
real symmetric matrices. For any $u=\sum_i u_i e_i\in TU$,
$$\aligned
(\mc{P}_0^\bot u)^T&=\sum_j \lan \mc{P}_0^\bot u,e_j\ran e_j=\sum_{i,j}\lan u_i(\mc{P}_0^\bot e_i),e_j\ran e_j\\
&=\sum_{i,j}u_i A_{ij}e_j.
\endaligned$$
Thus $(\mc{P}_0^\bot u)^T=\cos^2 \th(p)u$ if and only if $A(p)\xi=\cos^2\th(p)\xi$, with $\xi:=(u_1\ \cdots \ u_n)^T$.
Furthermore, $\th\in \Arg^T$ if and only if $\la(p):=\cos^2 \th(p)$ is a characteristic root of $A(p)$ for every $p\in U$,
and the multiplicity of $\la(p)$ equals $m_\th^T(p)$. If $m_\th^T\equiv k$, the main theorem in \cite{no}
enable us to find smooth mappings $\xi_1,\cdots,\xi_k:V\ra \R^n$, where $V$ is a neighborhood of $p_0$, such that for every $p\in V$,
$A(p)\xi_i=\la(p)\xi_i$ holds for $1\leq i\leq k$ and $\{\xi_1(p),\cdots,\xi_m(p)\}$ is linear independent. Denote
$\xi_i=(u_{i1},\cdots,u_{in})^T$ and let
$$X_i=\sum_j u_{ij}e_j,$$
then $X_1,\cdots,X_m$ are smooth tangent vector fields on $V$ and $T_\th M\cap T_p M$ is spanned by $X_1(p),\cdots,X_m(p)$ for any $p\in V$.
Finally the arbitrariness of $p_0$ ensures that $T_\th M$ is a smooth distribution on $M$.

\end{proof}

\bigskip\bigskip

\Section{Subharmonic functions}{Subharmonic functions}\label{pre}

Let $\bar{M}^{n+m}$ be a Riemannian manifold, and $M^n\rightarrow
\bar{M}^{n+m}$ be an isometric immersion.
The second fundamental form $B$ is a pointwise symmetric bilinear
form on $T_p M$ ($p\in M$) with values in $N_p M$ defined by
$$B_{XY}=(\overline{\n}_X Y)^N$$
with $\overline{\n}$ the Levi-Civita connection on $\bar{M}$. The
induced connections on $TM$ and $NM$ are defined by
$$\n_X Y=(\overline{\n}_X Y)^T,\qquad \n_X \nu=(\overline{\n}_X \nu)^N.$$
Here $X,Y$ are smooth sections of $TM$ and $\nu$ denotes a smooth
sections of $NM$. The second fundamental form, the curvature tensor
of the submanifold, the curvature tensor of the normal bundle and
the curvature tensor of the ambient manifold satisfy the Gauss
equations, the Codazzi equations and the Ricci equations (see
\cite{x} for details).

The trace of the second fundamental form gives a normal vector field
$H$ on $M$, which is called the mean curvature vector field. If $\n
H\equiv 0$, then we say $M$ has parallel mean curvature. Moreover if
$H\equiv 0$, $M$ is called a minimal submanifold in $\bar{M}$.

Now we consider an $n$-dimensional oriented submanifold $M$ in $\R^{n+m}$.
 Let $\{e_1,\cdots,e_n\}$ be a local orthonormal tangent frame field,
$\{\nu_1,\cdots,\nu_m\}$ be a local orthonormal  normal frame field
on $M$, and
$$h_{\a,ij}:=\lan B_{e_i e_j},\nu_\a\ran$$
be the coefficients of the second fundamental form $B$. We shall use
the summation convention and agree on the ranges of indices
$$1\leq i,j,k,l\leq n,\qquad 1\leq \a,\be,\g,\de\leq m.$$

The normal Gauss map
$\g: M\ra \grs{m}{n}$ is defined by
$$\g(p)=N_p M\in \grs{m}{n}$$
via parallel translation in $\R^{n+m}$ for every $p\in M$.
Let $Q_0$ be a fixed point in $\grs{m}{n}$ and
define
\begin{equation}
\aligned w:=w(\cdot,Q_0)\circ \g =\lan \nu_1\w\cdots\w \nu_m,Q_0\ran.
\endaligned
\end{equation}
 By the Codazzi equations,
it is not hard to get basic formulas for the function $w$ as
follows.

\begin{pro}(\cite{fc}\cite{x1})
Let $M$ be a submanifold in $\R^{n+m}$, then
\begin{equation}\label{dw}
\n_{e_i} w=-h_{\a,ij}\lan \nu_{\a j},Q_0\ran
\end{equation}
with
\begin{equation}
\nu_{\a j}:=\nu_1\w\cdots \w e_j\w \cdots\w \nu_m
\end{equation}
that is obtained by replacing $\nu_\a$ by $e_j$  in $\nu_1\w\cdots\w
\nu_m$. Moreover if $M$ has parallel mean curvature, then
\begin{equation}\label{La}
\De w=-|B|^2 w+\sum_i\sum_{\a\neq \be,j\neq k}h_{\a,ij} h_{\be,ik}
\lan \nu_{\a j,\be k},Q_0\ran.
\end{equation}
with
\begin{equation}
\nu_{\a j,\be k}:=\nu_1\w\cdots\w e_j \w\cdots\w e_k\w \cdots\w
\nu_m
\end{equation}
that is obtained by replacing  $\nu_\a$ by $e_j$  and $\nu_\be$ by
$e_k$ in $\nu_1\w\cdots\w \nu_m$, respectively.

\end{pro}

Let $p\in M$ satisfying $w(p)=w(N_p M,Q_0)>0$.
Denote by
$$\pi/2>\th_1\geq \cdots\geq \th_r>\th_{r+1}=\cdots=\th_m=0$$
the Jordan angles between $N_p M$ and $Q_0$.
 By Lemma \ref{basis3}, one can find
an oriented orthonormal basis $\{\ep_1,\cdots,\ep_m\}$ of
$Q_0$, an oriented orthonormal basis $\{\nu_1,\cdots,\nu_m\}$ of $N_p M$ and
an orthonormal basis $\{e_1,\cdots,e_n\}$ of $T_p M$, such that
$\mc{P}_0\nu_\a=\cos\th_\a \ep_\a$ and $\mc{P}_0 e_i=-\sin\th_i \ep_i$.
Especially
\begin{equation}
e_\a=\Phi_{\th_\a}(\nu_\a)\qquad \forall 1\leq \a\leq r.
\end{equation}

Putting
\begin{equation}\label{la}
\la_\a:=\tan\th_\a,
\end{equation}
then
\begin{equation}\label{w1}\aligned
\lan \nu_{\a j},Q_0\ran&=\lan \nu_1\w\cdots\w e_j\w\cdots\w \nu_m,\ep_1\w \cdots\w \ep_m\ran\\
&=\lan \mc{P}_0 \nu_1\w \cdots\w \mc{P}_0 e_j\w \cdots\w \mc{P}_0\nu_m,\ep_1\w \cdots\w \ep_m\ran\\
&=\lan \cos\th_1\ep_1\w \cdots\w (-\sin\th_j \ep_j)\w\cdots\w \cos\th_m\ep_m,\ep_1\w \cdots\w \ep_m\ran\\
&=-\de_{\a j}\tan\th_\a\big(\prod_{\g} \cos\th_\g)\\
&=-\de_{\a j}\la_\a w(p)
\endaligned
\end{equation}
and
\begin{equation}\label{w2}\aligned
&\lan \nu_{\a j,\be k},Q_0\ran\\
=&\lan \nu_1\w\cdots\w e_j\w\cdots\w e_k\w\cdots\w \nu_m,\ep_1\w\cdots\w \ep_m\ran\\
=&\lan \mc{P}_0 \nu_1\w\cdots \w \mc{P}_0 e_j\w\cdots\w \mc{P}_0 e_k\w\cdots\w \mc{P}_0 \nu_m,\ep_1\w \cdots\w \ep_m\ran\\
=&\lan \cos\th_1\ep_1\w \cdots\w (-\sin\th_j\ep_j)\w\cdots\w (-\sin\th_k \ep_k)\w\cdots\w \cos\th_m\ep_m,\ep_1\w\cdots\w \ep_m\ran\\
=&(\de_{\a j}\de_{\be k}-\de_{\a k}\de_{\be j})\tan\th_\a\tan\th_\be\big(\prod_{\g}\cos\th_\g)\\
=&(\de_{\a j}\de_{\be k}-\de_{\a
k}\de_{\be j})\la_\a\la_\be w(p).
\endaligned
\end{equation}
Substituting (\ref{w1}) and (\ref{w2}) into (\ref{dw}) and
(\ref{La}), respectively, we obtain
\begin{equation}
w^{-1}\n_{e_i}w=\sum_\a \la_\a h_{\a,i\a},
\end{equation}
\begin{equation}
w^{-1}\De w=-|B|^2+\sum_i\sum_{\a\neq
\be}\la_\a\la_\be(h_{\a,i\a}h_{\be,i\be}-h_{\a,i\be}h_{\be,i\a}).
\end{equation}

Let
\begin{equation}\label{v}
v:=w^{-1},
\end{equation}
then
\begin{equation}
\aligned
v^{-1}\De v&=w\big(-w^{-2}\De w+2w^{-3}|\n w|^2\big)\\
&=|B|^2+2\sum_{i,\a}\la_\a^2 h_{\a,i \a}^2+\sum_i\sum_{\a\neq
\be}\la_\a\la_\be(h_{\a,i\a}h_{\be,i\be}+h_{\a,i\be}h_{\be,i\a})
\endaligned
\end{equation}

Observing that $\la_\a\la_\be=0$ whenever $\a>r$ or $\be>r$, we group the terms of (\ref{La2}) according to the different types
of the indices of the coefficients of the second fundamental form as
follows.
\begin{equation}\label{La2}\aligned
v^{-1}\De v=&\sum_\a\sum_{i,j>r}h_{\a,ij}^2+\sum_{i>r}I_i+\sum_{i>r,1\leq \a<\be\leq r}II_{i\a\be}\\
            &+\sum_{1\leq \a<\be< \g\leq r}III_{\a\be\g}+\sum_{1\leq \a\leq r}IV_\a
            \endaligned
\end{equation}
where
\begin{equation}\label{I}
I_i=\sum_{1\leq \a\leq r} (2+2\la_\a^2)h_{\a,i\a}^2+\sum_{1\leq
\a,\be\leq r,\a\neq \be}\la_\a\la_\be h_{\a,i\a}h_{\be,i\be},
\end{equation}
\begin{equation}\label{II}
II_{i\a\be}=2h_{\a,i\be}^2+2h_{\be,i\a}^2+2\la_\a\la_\be
h_{\a,i\be}h_{\be,i\a},
\end{equation}
\begin{equation}\aligned
III_{\a\be\g}=&2h_{\a,\be\g}^2+2h_{\be,\g\a}^2+2h_{\g,\a\be}^2\\
&+2\la_\a\la_\be h_{\a,\be\g}h_{\be,\g\a}+2\la_\be\la_\g
h_{\be,\g\a}h_{\g,\a\be}+2\la_\g\la_\a h_{\g,\a\be}h_{\a,\be\g}
\endaligned
\end{equation}
and
\begin{equation}\aligned
IV_\a=&(1+2\la_\a^2)h_{\a,\a\a}^2+\sum_{1\leq \be\leq r, \be\neq \a}\big(h_{\a,\be\be}^2+(2+2\la_\be^2)h_{\be,\a\be}^2\big)\\
&+\sum_{1\leq \be,\g\leq r,\be\neq \g}\la_\be \la_\g
h_{\be,\a\be}h_{\g,\a\g}+2\sum_{1\leq \be\leq r,\be\neq
\a}\la_\a\la_\be h_{\a,\be\be}h_{\be,\a\be}.
\endaligned
\end{equation}
Note that the first term vanishes whenever $n=r$, the second
term vanishes whenever $n=r$ or $r\leq 1$ and the third term
vanishes whenever $r\leq 2$.

It is easily seen that
\begin{equation}
I_i=\big(\sum_{1\leq\a\leq r} \la_\a h_{\a,i\a}\big)^2+\sum_{1\leq
\a\leq r} (2+\la_\a^2)h_{\a,i\a}^2\geq 2\sum_{1\leq \a\leq
r}h_{\a,i\a}^2.
\end{equation}
Hence $I_i=0$ if and only if $h_{\a,i\a}=0$ for any $1\leq \a\leq
r$.

\bigskip
\begin{lem}\label{l2}
If $v(p)\leq 3$, then $II_{i\a\be}\geq 0$ for any $i>r$, $1\leq \a<\be\leq r$, and $\sum_{i>r,1\leq \a<\be\leq r}II_{i\a\be}=0$ if
and only if one of the following two cases occurs: (a)
$h_{\a,i\be}=h_{\be,i\a}=0$ for every $i,\a,\be$; (b) $r=2$, $\la_1=\la_2=\sqrt{2}$ and
$h_{1,i2}=-h_{2,i1}$ for each $i$.

\end{lem}

\begin{proof}
In conjunction with (\ref{v}), (\ref{w}) and (\ref{la}) we have
$$\aligned
v^2&=w^{-2}=\prod_\a \sec^2 \th_\a=\prod_{1\leq \a\leq r}(1+\tan^2 \th_\a)\\
&=\prod_{1\leq \a\leq r}(1+\la_\a^2)\geq (1+\la_1^2)(1+\la_2^2)\\
&=1+\la_1^2\la_2^2+\la_1^2+\la_2^2\geq 1+\la_1^2\la_2^2+2\la_1\la_2\\
&=(1+\la_1\la_2)^2\geq (1+\la_\a\la_\be)^2.
\endaligned$$
Thus $v\leq 3$ implies $\la_\a\la_\be\leq 2$, and equality holds
if and only if $r=2$ and $\la_1=\la_2=\sqrt{2}$.

By (\ref{II}),
\begin{equation}
II_{i\a\be}=\la_\a\la_\be
(h_{\a,i\be}+h_{\be,i\a})^2+(2-\la_\a\la_\be)(h_{\a,i\be}^2+h_{\be,i\a}^2)\geq
0.
\end{equation}
Equality holds whenever one and only one of the two cases occurs: (i)
$h_{\a,i\be}=h_{\be,i\a}=0$; (ii) $\la_\a\la_\be=2$ and
$h_{\a,i\be}=-h_{\be,i\a}\neq 0$. Thereby the conclusion follows.

\end{proof}

\bigskip

\begin{lem}\label{l3}
If $v(p)\leq 3$, then $III_{\a\be\g}\geq 0$ for any $1\leq \a<\be<\g\leq r$ and the equality holds
if and only if $h_{\a,\be\g}=h_{\be,\g\a}=h_{\g,\a\be}=0$.
\end{lem}

\begin{proof}
Let
\begin{equation}
x:=h_{\a,\be\g},\ y:=h_{\be,\g\a},\ z:=h_{\g,\a\be}
\end{equation}
and
\begin{equation}
a:=\la_\a,\ b:=\la_\be,\ c:=\la_\g.
\end{equation}
Then it is sufficient to show that the quadratic form
\begin{equation}
III_{\a\be\g}=2x^2+2y^2+2z^2+2abxy+2bcyz+2cazx
\end{equation}
is positive definite whenever $a\geq b\geq c>0$ and
\begin{equation}\label{abc}\aligned
(1+a^2)(1+b^2)(1+c^2)&=(1+\la_\a^2)(1+\la_\be^2)(1+\la_\g^2)\\
&\leq \prod_{1\leq \de\leq r}(1+\la_\de^2)=v(p)^2\leq 9.
\endaligned
\end{equation}

It is easily-seen that
\begin{equation}
III_{\a\be\g}=(ax+by+cz)^2+(2-a^2)x^2+(2-b^2)y^2+(2-c^2)z^2.
\end{equation}

If $a^2\leq 2$, then (\ref{abc}) implies $2>b^2\geq c^2>0$. Hence
$III_{\a\be\g}\geq 0$ and equality holds if and only if
$ax+by+cz=y=z=0$, which is equivalent to $x=y=z=0$.

If $a^2>2$, again using (\ref{abc}) gives $2>b^2\geq c^2>0$. Putting
$s:=by+cz$, then by the Cauchy-Schwarz inequality,
\begin{equation}
\aligned
s^2&=(by+cz)^2\\
   &=\Big(\f{b}{\sqrt{2-b^2}}\sqrt{2-b^2}y+\f{c}{\sqrt{2-c^2}}\sqrt{2-c^2}z\Big)^2\\
   &\leq \Big(\f{b^2}{2-b^2}+\f{c^2}{2-c^2}\Big)\Big[(2-b^2)y^2+(2-c^2)z^2\Big]
   \endaligned
\end{equation}
and equality holds if and only if
$\big(\f{b}{\sqrt{2-b^2}},\f{c}{\sqrt{2-c^2}}\big)$ and
$(\sqrt{2-b^2}y,\sqrt{2-c^2}z)$ is linear independent, i.e.
\begin{equation}\label{yz}
\f{2-b^2}{b}y=\f{2-c^2}{c}z.
\end{equation}
Hence
\begin{equation}\aligned\label{III}
III_{\a\be\g}&=(ax+s)^2+(2-a^2)x^2+(2-b^2)y^2+(2-c^2)z^2\\
&\geq (ax+s)^2+(2-a^2)x^2+\Big(\f{b^2}{2-b^2}+\f{c^2}{2-c^2}\Big)^{-1}s^2\\
&=2x^2+2axs+\Big[1+\Big(\f{b^2}{2-b^2}+\f{c^2}{2-c^2}\Big)^{-1}\Big]s^2.
\endaligned
\end{equation}

It is well-known that the quadratic form $\sum_{1\leq i,j\leq
2}a_{ij}u_{ij}$ is positive definite if and only if $a_{11}>0$ and
$\det(a_{ij})>0$. Therefore, to show that the right hand side of
(\ref{III}) is positive definite, it is sufficient to prove
$$2\Big[1+\Big(\f{b^2}{2-b^2}+\f{c^2}{2-c^2}\Big)^{-1}\Big]-a^2>0,$$
which is equivalent to
\begin{equation}
\f{1}{2-a^2}+\f{1}{2-b^2}+\f{1}{2-c^2}<1.
\end{equation}

Denoting $u:=1+a^2, v:=1+b^2, w:=1+c^2$,
$$\Om=\{(u,v,w)\in \R^3: u>3>v\geq w>1,uvw\leq 9\}$$
and
$$f(u,v,w)=\f{1}{3-u}+\f{1}{3-v}+\f{1}{3-w}=\f{1}{2-a^2}+\f{1}{2-b^2}+\f{1}{2-c^2},$$
then it suffices to show $f|_\Om<1$. For a sufficiently small
positive constant $\ep$, let
$$\Om_{\ep}=\{(u,v,w)\in \R^3:u\geq 3+\ep,3-\ep\geq v,w\geq 1+\ep,uvw\leq 9\},$$
then $\Om_{\ep}$ is compact and there exists $(u_0,v_0,w_0)\in
\Om_{\ep}$, such that
\begin{equation}\label{max}
f(u_0,v_0,w_0)=\max_{\Om_{\ep}}f.
\end{equation}
Fix $v_0$, then (\ref{max}) implies for any $(u,v_0,w)\in
\Om_{\ep}$ such that $uw=u_0 w_0$,
$$f_{v_0}(u,w):=\f{1}{3-u}+\f{1}{3-v_0}+\f{1}{3-w}\leq \f{1}{3-u_0}+\f{1}{3-v_0}+\f{1}{3-w_0}.$$
Differentiating both sides of $uw=u_0w_0$ yields
$\f{du}{u}+\f{dw}{w}=0$. Hence
$$\aligned
df_{v_0}&=\f{du}{(3-u)^2}+\f{dw}{(3-w)^2}\\
&=\Big[-\f{u}{(3-u)^2}+\f{w}{(3-w)^2}\Big]\f{dw}{w}\\
&=\f{(w-u)(9-uw)}{(3-u)^2(3-w)^2}\f{dw}{w}
\endaligned$$
which means that $f_{v_0}(u,w)$ is strictly decreasing in $w$. Therefore
$w_0=1+\ep$. Similarly one can derive $v_0=1+\ep$. Therefore
$u_0=9(1+\ep)^{-2}$ and
$$\aligned
\max_{\Om_\ep}f&=f(u_0,v_0,w_0)=\f{1}{3-9(1+\ep)^{-2}}+\f{1}{3-(1+\ep)}+\f{1}{3-(1+\ep)}\\
&<\f{1}{3-9}+\f{1}{3-1}+\f{1}{3-1}=\f{5}{6}<1.
\endaligned$$
Then $f|_\Om<1$ follows from $\Om\subset \bigcup_{\ep}\Om_\ep$.

Thus $III_{\a\be\g}\geq 0$ and  equality holds if and only if
$x=0$, $s=by+cz=0$ and (\ref{yz}) holds true, which imply $x=y=z=0$
and the conclusion follows.

\end{proof}

In \cite{j-x-y2}, we obtained an estimate for the fourth term as
follows.

\begin{lem}\label{IV}\cite{j-x-y2}
 There exists a positive constant $\ep_0$ with the following property. If $v(p)\leq 3$, then
$$IV_\a\geq \ep_0\big(h_{\a,\a\a}^2+\sum_{1\leq \be\leq r,\be\neq \a}(h_{\a,\be\be}^2+2h_{\be,\a\be}^2)\big).$$
\end{lem}

In conjunction with (\ref{La2}), (\ref{I}), Lemma \ref{l2}-\ref{IV}, $\De v(p)\geq 0$ whenever $v(p)\leq 3$. Moreover if $\De v(p)=0$ and
$|B|^2(p)>0$, then $r=2$, $\la_1=\la_2=\sqrt{2}$, $h_{1,i2}=-h_{2,i1}$ for any $i\geq 3$ and the coeffients of $B$ belonging to the other types  all vanish.
In other words, $\Arg(N_p M,Q_0)\subset \{0,\th_0\}$, $\Arg(T_p M,Q_0^\bot)\subset \{0,\th_0\}$ and the multiplicity of the Jordan angle $\th_0$ equals $2$, where $\th_0:=\arctan \sqrt{2}$.
Let $N_{p,\th_0}M$, $N_{p,0}M$ ($T_{p,\th_0}M$, $T_{p,0}M$) denote angle spaces of $N_p M$
($T_p M$) relative to $Q_0$ ($Q_0^\bot$), then $N_{p,\th_0}M=\text{span}\{\nu_1,\nu_2\}$ and $T_{p,\th_0}M=\text{span}\{e_1,e_2\}$
with $e_\a=\Phi_{\th_0}(\nu_\a)$ for each $\a=1,2$, $N_{p,0}M=\text{span}\{\nu_3,\cdots,\nu_m\}$
and $T_{p,0}M=\text{span}\{e_3,\cdots,e_n\}$.

Denote
\begin{equation}
S_{\mu\nu}(v):=\lan B_{v,\Phi_{\th_0}(\mu)},\nu\ran\qquad \forall \mu,\nu\in N_{p,\th_0}M, v\in T_{p,0}M,
\end{equation}
Then $(\mu,\nu)\mapsto S_{\mu\nu}$ is a $T_{p,0}^*M$-valued bilinear form on $N_{p,\th_0}M$, where $V^*$
denotes the dual space of $V$. In other words, $S\in N_{p,\th_0}^* M\otimes N_{p,\th_0}^* M\otimes T_{p,0}^*M$. For
any $1\leq \a,\be\leq 2$ and $3\leq i\leq n$,
$$S_{\nu_\a \nu_\be}(e_i)=\lan B_{e_i,\Phi_\th(\nu_\a)},\nu_\be\ran=\lan B_{e_i e_\a},\nu_\be\ran=h_{\be,i\a},
$$
which implies $S_{\nu_1\nu_1}=S_{\nu_2\nu_2}=0$ and $S_{\nu_1\nu_2}=-S_{\nu_2\nu_1}\neq 0$,
 i.e. $S$ is antisymmetric.

Moreover,
observing that $h_{\a,ij}=0$ whenever $j\geq 3$ or $\a\geq 3$, $h_{\a,\be\g}=0$ for any $1\leq \a,\be,\g\leq 2$, we
can derive the following result.

\begin{pro}\label{sub-har}
Let $M^n$ be a submanifold in $\R^{n+m}$ with parallel mean curvature, $\g:M\ra \grs{m}{n}$ be the normal Gauss map and $Q_0$ be a fixed point in
$\grs{m}{n}$. Put $v:=w^{-1}(\cdot,Q_0)\circ \g$, then for any $p\in M$, $v(p)\leq 3$ implies $\De v\geq 0$ at $p$. Furthermore, if $\De v(p)=0$,
then one and only one of following 2 cases must occur: (a) $|B|^2(p)=0$. (b) $|B|^2(p)\neq 0$, $\Arg(N_p M,Q_0)\subset \{0,\th_0\}$
and $\Arg(T_p M,Q_0^\bot)\subset \{0,\th_0\}$ with $\th_0:=\arctan \sqrt{2}$, and the multiplicity of $\th_0$ is $2$; let $N_{p,\th_0}M, N_{p,0}M, T_{p,\th_0}M, T_{p,0}M$ be angle spaces and
$\Phi_{\th_0}$ denote the anti-involutive automorphism associated to $\th_0$,
then there exists a nonzero  $T_{p,0}^*M$-valued antisymmetric bilinear form $S$ on $N_{p,\th_0}M$, such that
\begin{equation}\label{sec}
B=\sum_{1\leq \a,\be\leq 2}(S_{\nu_\a\nu_\be}\odot \om_\a)\otimes \nu_\be,
\end{equation}
where $\{\nu_1,\nu_2\}$ is an arbitrary orthonormal basis of $N_{p,\th_0}M$, $\om_\a(v):=\lan v,\Phi_{\th_0}(\nu_\a)\ran$
for every $v\in T_p M$ and $\om\odot \si:=\om\otimes \si+\si\otimes \om$.
\end{pro}

{\bf Remark.} Let $V$ be an $n$-dimensional real linear space equipped with an inner product $\lan\cdot,\cdot\ran$. A linear subspace
$\mc{Q}\subset S^2(V^*)$ of the quadratic functions on $V$ is said to be \textit{austere} if the odd symmetric functions of the eigenvalues of
any element of $\mc{Q}$ with respect to $\lan\cdot,\cdot\ran$ are all zero, i.e. the nonzero eigenvalues occur in  pairs of opposite signs.
Let $\mc{Q}$ be an austere subspace of $S^2(V^*)$, if there is a nonzero vector $v_0\in V$, such that $\varphi(v_0,v_0)=0$ and $\varphi(v,w)=0$
for any $\varphi\in \mc{Q}$ and $v,w\in v_0^\bot$, and $\dim\mc{Q}\geq 2$, then we say $\mc{Q}$ is \textit{simple} (see \cite{br}).  Let $M^n$ be a submanifold
of $\R^{n+m}$, then for any $p\in M$ and $\nu\in N_p M$, $B_p^\nu:(v,w)\mapsto \lan B_{v,w},\nu\ran$
is a quadratic function on $T_p M$ and hence $\mc{B}_p:=\{B_p^\nu:\nu\in N_p M\}$ is a linear subspace of $S^2(T_p^* M)$. If $\mc{B}_p$
is (simple) austere for every $p\in M$, then $M$ is said to be a (\textit{simple}) \textit{austere submanifold}. The concept of austere submanifolds
was introduced by Harvey-Lawson \cite{h-l} in connection with their foundational work on calibrations.
%Bryant \cite{br} gave
%a structure theorem for simple austere submanifolds, saying that any connected simple austere submanifold is congruent to an open subset of
%a \textit{generalized helicoids}.
Assume $M$ to be a submanifold with parallel mean curvature.
Let $p\in M$, if $v(p)\leq 3$, $\De v(p)=0$ and $|B|^2(p)\neq 0$, then $\Arg(N_p M,Q_0)\subset \{0,\th_0\}$
with $\th_0:=\arctan \sqrt{2}$ and the multiplicity of $\th_0$ is $2$. Now we choose $v_0\in T_{p,0}M$, such that $\lan v_0,v\ran=S_{\nu_1\nu_2}(v)$ for
any $v\in T_{p,0}M$, then (\ref{sec})
implies $\lan B_{v_0v_0},\nu\ran=0$ and $\lan B_{vw},\nu\ran=0$ for any $v,w\in v_0^\bot$, i.e. $\mc{B}_p$ is a simple austere subspace
of $S^2(T^*M)$. Moreover, if $v\leq 3$ and $\De v\equiv 0$ on $M$, then $M$
is a simple austere submanifold, which is congruent to an open subset of a \textit{generalized helicoids}, due to
Bryant's structure theorem for simple austere submanifolds (see \cite{br}).

\bigskip\bigskip

\Section{Bernstein theorems}{Bernstein theorems}\label{be}

We will primarily study a submanifold $M^{n-1}$ in $S^{n+m-1}$, the
standard unit sphere in $\R^{n+m}$. The cone $CM$ over $M$ is the
image of the map $M\times [0,\infty)\rightarrow \R^{n+m}$ defined by
$(\mathbf{x},t)\mapsto t\mb{x}$, namely
\begin{equation}
CM=\{t\mathbf{x}\in \R^{n+m}: t\in [0,\infty),\mb{x}\in M\}.
\end{equation}
Obviously $CM$ has a singularity $t=0$ unless $M$ is a subsphere. To
avoid the singularity we consider the truncated cone $CM_\ep$
defined by
\begin{equation}
CM_\ep=\{t\mb{x}\in \R^{n+m}:t\in (\ep,\infty),\mb{x}\in M\}
\end{equation}
with $\ep>0$.

$M$ and $CM_\ep$ share similar geometric properties. At First,
the comparison of the second fundamental form $B$ of $M$ and $B^c$ of
$CM_\ep$ immediately yields the following result.

\bigskip
\begin{pro}\label{min}(\cite{x0} p.64)
If $CM_\ep$ has parallel mean curvature in $\R^{n+m}$, then $M$ is a
minimal submanifold in $S^{n+m-1}$. Conversely, if $M$ is a minimal
submanifold in $S^{n+m-1}$, then $CM_\ep$ is a minimal submanifold
in $\R^{n+m}$.
\end{pro}
\bigskip

There is a natural map from $\R^{n+m}/\{0\}$ to $S^{n+m-1}$ by
\begin{equation}
\psi(\mb{x})=\f{\mb{x}}{|\mb{x}|}.
\end{equation}
Hence for any map $F_1$ from $M$ to an arbitrary Riemannian
manifold, the map
\begin{equation}
F:=F_1\circ \psi
\end{equation}
on $CM_\ep$ is called a \textit{cone-like map} (see \cite{x0} p.66). A direct
calculation shows that $F$ is a harmonic map if and only if $F_1$ is
harmonic (\cite{x0} p.67). Especially, when $F_1$ is a function, $F$
is a harmonic (subharmonic, superharmonic) function if and only if
$F_1$ is harmonic (subharmonic, superharmonic).

For any $p\in M$, $N_p M\subset T_p S^{n+m-1}$ can be viewed as an
$m$-dimensional affine subspace in $\R^{n+m}$. Via parallel
translation in the Euclidean space, one can define the normal Gauss map
$\g: M\ra \grs{m}{n}$
\begin{equation}
p\mapsto N_p M.
\end{equation}

The canonical normal Gauss map on $CM_\ep\subset \R^{n+m}$ is
defined by
\begin{equation}
\g^c: p\in CM_\ep\mapsto N_p (CM_\ep)\in \grs{n}{m}.
\end{equation}
It is easily-seen that $\g^c$ is a cone-like map; more precisely,
\begin{equation}
\g^c=\g\circ \psi.
\end{equation}
Therefore, the Gauss image of $CM_\ep$ coincides with the Gauss image of $M$.
Applying Proposition \ref{sub-har}, we can derive the following spherical Bernstein theorem.

\begin{thm}\label{ber}
Let $M^{n-1}$ be a compact, oriented minimal submanifold in
$S^{n+m-1}$. If there is a fixed oriented $m$-plane $Q_0$, such that
\begin{equation}\label{con}
\lan N,Q_0\ran \geq 1/3
\end{equation}
for all normal $m$-planes $N$ of $M$, then $M$ is a totally geodesic
subsphere of $S^{n+m-1}$.

\end{thm}

\begin{proof}
Let $CM_\ep$ be the truncated cone generated by $M$, then by
Proposition \ref{min}, $CM_\ep$ is an $n$-dimensional minimal
submanifold in $\R^{n+m}$. Denote by $\g^c$ the normal Gauss map of
$M$, and
\begin{equation}
v:=w^{-1}(\cdot,Q_0)\circ \g^c.
\end{equation}
Since $\g^c$ is a cone-like map extended by the normal Gauss map
$\g: M\ra \grs{m}{n}$, $v$ is a cone-like function. Hence the
condition (\ref{con}) implies $v\leq 3$ everywhere on $CM_\ep$ and
by applying Proposition \ref{sub-har} we know $v$ is a subharmonic
function on $CM_\ep$. Since $v$ is a cone-like function, $v|_M$ is
also a subharmonic function on $M$. The classical maximum principle
implies $v|_M$ is constant, therefore $\De v\equiv 0$ on $CM_\ep$.

\begin{equation}
U:=\{p\in CM_\ep:|B|^2(p)>0\}
\end{equation}
is an open subset of $CM_\ep$. It suffices to show $U=\emptyset$.

We prove it by a reductio ad absurdum. Assume $U\neq \emptyset$.
Let $\th$ be a normal (tangent) Jordan angle function on $U$,
then Proposition \ref{sub-har} implies $\th(p)=\th_0$ or $0$
for any $p\in U$, where $\th_0:=\arctan \sqrt{2}$. The continuity of $\th$
forces $\th\equiv \th_0$ or $0$. Moreover, since $m_{\th_0}^N=m_{\th_0}^T\equiv 2$,
Lemma \ref{dis} implies $N_{\th_0} U$, $N_0 U$ are smooth subbundles of $NU$, and $T_{\th_0} U$, $T_0 U$ are smooth subbundles
of $TU$.
It is worthy to note that, if $m=2$ ($n=2$), $\nu\in N_0 U$ ($u\in T_0 U$) if and only if $\nu=0$ ($u=0$).
Again applying Proposition \ref{sub-har}, the normal bundle and the tangent bundle have the following decomposition
$$\aligned
NU&=N_{\th_0} U\oplus N_0 U,\\
TU&=T_{\th_0} U\oplus T_0 U
\endaligned$$
and there exists $S\in \G\big( \La^2(N_{\th_0}^* U)\otimes T_0^* U\big)$, which is nonzero everywhere on $U$, such that for any $p_0\in U$,
\begin{equation}\label{S}
B=\sum_{1\leq \a,\be\leq 2}(S_{\nu_\a\nu_\be}\odot \om_\a)\otimes \nu_\be
\end{equation}
holds on a neighborhood of $p_0$, where $\{\nu_1,\nu_2\}$ is a local orthonormal frame field of $N_{\th_0} U$ and
$\om_\a(v):=\lan v,\Phi_{\th_0}(\nu_\a)\ran$. Now we put
$e_\a(p):=\Phi_{\th_0}\big(\nu_\a(p)\big)$ for every $p$ and each $\a=1,2$, then $\{e_1,e_2\}$ is a local orthonormal frame field of $T_{\th_0} U$.
Since $S$ is nonzero everywhere, there are a unit vector field $e_n$ of $T_0 U$, and a positive function $h$, such that
\begin{equation}\label{h}
S_{\nu_1\nu_2}(v)=h\lan v,e_n\ran\qquad \forall v\in T_0 U.
\end{equation}
Now we choose $e_3,\cdots,e_{n-1}$ ($\nu_3,\cdots,\nu_m$) to be unit vector fields of $T_0 U$ ($N_0 U$), such that
$\{e_1,\cdots,e_n\}$ ($\{\nu_1,\cdots,\nu_m\}$) is a local orthonormal tangent (normal) frame field around $p_0$.
Then (\ref{S}) and (\ref{h}) shows
$B_{e_i e_j}=0$ for any $1\leq i,j\leq n-1$,
$B_{e_n e_i}=0$ for every $i\geq 3$, $B_{e_n e_1}=h \nu_2$ and $B_{e_n e_2}=-h\nu_1$.

 Noting that $e_n\in Q_0^\bot$ everywhere, we have
\begin{equation}\label{orthogonal}
\overline{\n}_{e_i}e_n\in Q_0^\bot\qquad \forall 1\leq i\leq n
\end{equation}
with $\overline{\n}$ the Levi-Civita connection on $\R^{n+m}$.

Denote
\begin{equation}
\G_{ij}^n:=\lan \n_{e_i}e_j,e_n\ran\qquad \forall 1\leq i,j\leq n-1.
\end{equation}
Differentiating both sides of $B_{e_i e_j}\equiv 0$ gives
\begin{equation}\label{Ga1}
\aligned
0&=\n_{e_k}B_{e_i e_j}=(\n_{e_k}B)_{e_i e_j}+B_{\n_{e_k}e_i,e_j}+B_{e_i,\n_{e_k}e_j}\\
&=(\n_{e_k}B)_{e_i e_j}+\G_{ki}^n B_{e_n e_j}+\G_{kj}^n B_{e_i e_n}
\endaligned
\end{equation}
Interchanging the position of $i$ and $k$ in the above formula yields
\begin{equation}\label{Ga2}
(\n_{e_i}B)_{e_k e_j}+\G_{ik}^n B_{e_n e_j}+\G_{ij}^n B_{e_k e_n}=0.
\end{equation}
The well-known Codazzi equations tell us $(\n_{e_k}B)_{e_i e_j}=(\n_{e_i}B)_{e_k e_j}$, then combining (\ref{Ga1}) and (\ref{Ga2}) gives
\begin{equation}\label{Gamma}
\G_{ki}^n B_{e_n e_j}+\G_{kj}^n B_{e_i e_n}-\G_{ik}^n B_{e_n e_j}-\G_{ij}^n B_{e_k e_n}=0
\end{equation}
holds for any $1\leq i,j,k\leq n-1$. For any $3\leq i,j\leq n-1$, choosing $k=1$ in (\ref{Gamma}) implies
$0=-\G_{ij}^n B_{e_1 e_n}=-h\G_{ij}^n \nu_2$, hence $\G_{ij}^n=0$. For any $3\leq i\leq n-1$, taking $j=1,k=2$ in (\ref{Gamma}) gives
$$\aligned
0&=\G_{2i}^n B_{e_n e_1}-\G_{i2}^n B_{e_n e_1}-\G_{i1}^n B_{e_2 e_n}\\
&=h(\G_{2i}^n-\G_{i2}^n)\nu_2+h\G_{i1}^n \nu_1
\endaligned$$
and hence $\G_{2i}^n-\G_{i2}^n=\G_{i1}^n=0$. Similarly, choosing $j=2,k=1$ yields $\G_{1i}^n-\G_{i1}^n=\G_{i2}^n=0$.
Therefore $\G_{1i}^n=\G_{2i}^n=\G_{i1}^n=\G_{i2}^n=0$.  Now we put $i=j=1$, $k=2$, then (\ref{Gamma}) tells us
$$\aligned
0&=\G_{21}^n B_{e_n e_1}+\G_{21}^n B_{e_1 e_n}-\G_{12}^n B_{e_n e_1}-\G_{11}^n B_{e_2 e_n}\\
&=h(2\G_{21}^n-\G_{12}^n)\nu_2+h\G_{11}^n \nu_1
\endaligned$$
and hence $2\G_{21}^n-\G_{12}^n=\G_{11}^n=0$. Similarly, putting $i=j=2$ and $k=1$ in (\ref{Gamma}) yields
$2\G_{12}^n-\G_{21}^n=\G_{22}^n=0$. Therefore $\G_{11}^n=\G_{12}^n=\G_{21}^n=\G_{22}^n=0$. In summary
\begin{equation}
\lan \n_{e_i}e_j,e_n\ran=\G_{ij}^n=0\qquad \forall 1\leq i,j\leq n-1.
\end{equation}

Thus
\begin{equation}
\aligned
\overline{\n}_{e_1}e_n&=\sum_{1\leq j\leq n-1}\lan \n_{e_1}e_n,e_j\ran e_j+\lan \n_{e_1}e_n,e_n\ran+B_{e_1 e_n}\\
&=-\sum_{1\leq j\leq n-1}\lan \n_{e_1}e_j,e_n\ran e_j+B_{e_1 e_n}\\
&=h\nu_2
\endaligned
\end{equation}
and then by (\ref{orthogonal}), $\nu_2$ is orthogonal to $Q_0$, i.e. $\mc{P}_0 \nu_2=0$. But on the other hand, $\nu_2\in N_{\th_0} U$
implies $|\mc{P}_0\nu_2|=\cos\th_0$, which is a contradiction.

Therefore $U=\emptyset$ and $CM_\ep$ is totally geodesic in $\R^{n+m}$. Hence $M$ has to be a totally geodesic subsphere.

\end{proof}

With the results of geometric measure theory, we can also prove a Euclidean Bernstein type theorem.
\bigskip

\begin{thm}\label{t3}
Let $f:=(f^1,\cdots,f^m)$ be a smooth $\R^m$-valued function defined everywhere on $\R^n$. Suppose its graph
$M:=\text{graph }f=\{(x,f(x)):x\in \R^m\}$ is a minimal submanifold in $\R^{n+m}$, and
\begin{equation}\label{con2}
\De_f:=\Big[\det\Big(\de_{ij}+\sum_\a \f{\p f^\a}{\p x^i}\f{\p
f^\a}{\p x^j}\Big)\Big]^{\f{1}{2}}\leq 3,
\end{equation}
then $f^1,\cdots,f^m$ has to be affine linear,  representing an affine $n$-plane in $\R^{n+m}$.

\end{thm}

\begin{proof}
As shown in \cite{j-x-y2}, the condition (\ref{con2}) says
$$v:=w^{-1}(\cdot,Q_0)\circ \g\leq 3.$$
Here $\g$ denotes the normal Gauss map of $M$ into $\grs{m}{n}$, and $Q_0:=\ep_{n+1}\w\cdots\w \ep_{n+m}$,
where $\ep_1,\cdots,\ep_{n+m}$ denotes the Cartesian coordinate vectors in $\R^{n+m}$. In other words, the Gauss image
of $M$ is contained in a closed region
$$\ol{\Bbb{V}}:=\{P\in \grs{m}{n}: w(P,Q_0)\geq 1/3\}.$$
Now we consider the tangent cone of $M$ at $\infty$, which is the limit of a one-parameter family of minimal submanifolds
in $\R^{n+m}$; each one is obtained by a contracting procedure. More precisely, let
\begin{equation}
f_t=\f{1}{t}f(tx),\qquad \forall t\in \R^+,
\end{equation}
then it is easy to check that $f_t$ satisfies the minimal surface equations and hence $\{M_t=\text{graph }f_t:t\in \R^+\}$
defines a family of minimal submanifolds in $\R^{n+m}$. Based on (\ref{con2}), one can proceed as in \cite{fc} \S 5 to show that
$\{f_t:t\in \R^+\}$ is an equicontinuous family on any compact subset of $\R^n$. Thus the Arzela-Ascoli Theorem implies that
there exists a subsequence $\{f_{t_i}:i\in \Bbb{Z}^+\}$ such that $\lim_{i\ra \infty}t_i=+\infty$ and
$\lim_{i\ra \infty}f_{t_i}=h$, which is a Lipschitz function. One can
prove that $h$ is a weak solution to the minimal surface equations
and its graph $C(M,\infty):=\{(x,h(x)):x\in \R^n\}$ is a cone (see \cite{l}), which is called the tangent cone of $M$ at
$\infty$. In the framework of geometric measure theory, we can prove
that $C(M,\infty)$ is regular except at $0$, as in \cite{fc}.
Hence the intersection of $C(M,\infty)$ and the unit sphere gives an $(n-1)$-dimensional embedded minimal submanifold in
$S^{n+m-1}$, which is denoted by $M'$. The image of the normal Gauss map of $M'$ is still contained in $\ol{\Bbb{V}}$,
i.e. $\lan N,Q_0\ran\geq 1/3$ for all normal $m$-planes of $M'$ (the proof is the same as in \cite{x} \S 7.3). By Theorem
\ref{ber}, $M'$ has to be a totally geodesic subsphere. Then Allard's regularity estimate \cite{a} implies $f$ is affine linear
and $M$ is an affine $n$-plane.
\end{proof}

\bibliographystyle{amsplain}

\end{document}